\documentclass[a4paper]{amsart}
\usepackage{amsfonts}
\usepackage{graphicx}
\usepackage{mathrsfs}
\usepackage[margin=1.5in]{geometry}

\usepackage{amsmath}
\usepackage{amssymb}
\usepackage[small,nohug,heads=LaTeX]{diagrams}
\usepackage{pstricks}
\usepackage{verbatim}
\usepackage{enumerate}
\usepackage{mathtools}

\newtheorem{MainThm}{Theorem}

\newtheorem{thm}{Theorem}[section]

\newtheorem{lem}[thm]{Lemma}
\newtheorem{prop}[thm]{Proposition}
\theoremstyle{definition}

\theoremstyle{remark}
\newtheorem{rem}[thm]{Remark}

\numberwithin{equation}{section}

\newcommand\mnote[1]{\marginpar{\tiny #1}}

\newcommand{\bF}{\mathbb{F}}

\newcommand{\bL}{\mathbb{L}}

\newcommand{\bP}{\mathbb{P}}
\newcommand{\bQ}{\mathbb{Q}}
\newcommand{\bR}{\mathbb{R}}

\newcommand{\bZ}{\mathbb{Z}}

\newcommand\lra{\longrightarrow}
\newcommand\lla{\longleftarrow}
\newcommand\trf{\mathrm{trf}}

\newcommand\Emb{\mathrm{Emb}}

\newcommand\hocolim{\mathrm{hocolim \,}}
\newcommand\Coker{\mathrm{Coker}}
\newcommand\Ker{\mathrm{Ker}}

\newcommand{\hcoker}{/\!\!/}

\newarrow {Onto} ----{>>}
\newarrow {Equals} ===={=}

\newcommand{\CircNum}[1]{\ooalign{\hfil\raise .00ex\hbox{\scriptsize #1}\hfil\crcr\mathhexbox20D}}

\newcommand{\hocofib}{\mathrm{hocofib}}

\mathchardef\ordinarycolon\mathcode`\:
\mathcode`\:=\string"8000
\begingroup \catcode`\:=\active
  \gdef:{\mathrel{\mathop\ordinarycolon}}
\endgroup

\title{Homological stability for unordered configuration spaces}

\author{Oscar Randal-Williams}
\thanks{The author was supported by the EPSRC PhD Plus scheme, ERC Advanced Grant No.\ 228082, and the Danish National Research Foundation via the Centre for Symmetry and Deformation.}
\address{Institut for Matematiske Fag\\
Universitetsparken 5\\
DK-2100 K{\o}benhavn {\O}\\
Denmark}
\email{o.randal-williams@math.ku.dk}

\begin{document}

\begin{abstract}
This paper consists of two related parts. In the first part we give a self-contained proof of homological stability for the spaces $C_n(M;X)$ of configurations of $n$ unordered points in a connected open manifold $M$ with labels in a path-connected space $X$, with the best possible integral stability range of $2* \leq n$. Along the way we give a new proof of the high connectivity of the complex of injective words. If the manifold has dimension at least three, we show that in rational homology the stability range may be improved to $* \leq n$.

In the second part we study to what extent the homology of the spaces $C_n(M)$ can be considered stable when $M$ is a closed manifold. In this case there are no stabilisation maps, but one may still ask if the dimensions of the homology groups over some field stabilise with $n$. We prove that this is true when $M$ is odd-dimensional, or when the field is $\bF_2$ or $\bQ$. It is known to be false in the remaining cases.
\end{abstract}
\maketitle

\section{Introduction and statement of results}

We define the space of ordered configurations of $n$ points in a $d$-manifold $M$ with labels in a path-connected space $X$ by
$$\widetilde{C}_n(M;X) := \Emb(\{1, 2, ..., n\}, M) \times X^n,$$
and the space of unordered configurations by
$$C_n(M;X) := \widetilde{C}_n(M;X)/\Sigma_n = \Emb(\{1, 2, ..., n\}, M) \times_{\Sigma_n} X^n.$$
If $M$ is the interior of a manifold with boundary $\bar{M}$, for each point $\mathcal{E} \in \partial\bar{M}$ there is a map
$$s_\mathcal{E} : C_n(M;X) \lra C_{n+1}(M;X)$$
given by pushing configuration points away from the boundary and then adding a new point near $\mathcal{E}$. We will describe this construction precisely in \S \ref{sec:StabMaps}.

\begin{MainThm}\label{thm:Main}
Let $M$ be a connected open manifold which is the interior of a compact manifold with boundary and has dimension greater than 1. Then each map $s_\mathcal{E} : C_n(M ; X) \to C_{n+1}(M ; X)$ is split injective on integral homology in all degrees and an integral homology isomorphism in degrees $2* \leq n$.
\end{MainThm}

\begin{MainThm}\label{thm:MainRational}
If $M$ has dimension greater than $2$, each $s_\mathcal{E}$ is a rational homology isomorphism in degrees $* \leq n$.
\end{MainThm}

If $M$ is a closed manifold, there are no possible stabilisation maps as there is no distinguished place to add a new point. However one may still study the sequences $n \mapsto \mathrm{dim}_\bF H_i(C_n(M);\bF)$ and ask if they stabilise. In order that these dimensions be finite, we restrict to manifolds $M$ having finite-dimensional $\bF$-homology in each degree.

\begin{MainThm}\label{thm:MainClosed}
Suppose that $\dim(M)$ is odd and $\bF$ is any field, or that $\dim(M)$ is even and $\bF$ is $\bF_2$ or $\bQ$. Then $\mathrm{dim}_\bF H_i(C_n(M);\bF)$
is independent of $n$ in the range that $H_i(C_n(M\setminus *);\bF)$ is stable. This is false in the remaining cases.

Furthermore, over $\bQ$ the evident transfer map $H_*(C_n(M);\bQ) \to H_*(C_{n-1}(M);\bQ)$ induces an isomorphism.
\end{MainThm}

\begin{rem}
We do not claim any originality for Theorem \ref{thm:Main}, although it does not seem to have appeared in this level of generality before. In the case $X=*$, the fact that these spaces exhibit homological stability at all is due to McDuff \cite{McDuff} and a stability range similar to the above was later obtained by Segal \cite[Proposition A.1]{SegalRational}. For $M = \bR^d$ and $X$ arbitrary, the fact that these spaces exhibit homological stability follows from the work of F.\ Cohen \cite[\S III]{CLM}, who in fact completely computed the $\bF_p$-homology of $C_n(\bR^d;X)$. A geometric approach to stability in this case was given by Lehrer and Segal \cite[Theorem 3.2]{LehrerSegal}, though their methods are somewhat indirect and so unable to provide a stability range. Our improved stability range for rational homology for manifolds of dimension at least three in Theorem \ref{thm:MainRational} seems to be new.

Theorem \ref{thm:MainClosed} for $\dim(M)$ odd or over $\bF_2$ was originally obtained by B{\"o}digheimer, Cohen and Taylor \cite{BCT} and over $\bF_2$ was independently obtained by Milgram and L{\"o}ffler \cite{ML}. It is the case of a general closed manifold and the field $\bQ$ that is our main contribution. 

T.\ Church has a preprint \cite{ChurchConf} in which he establishes rational homological stability in the representational sense (for the definition of which we refer to his paper with Farb \cite{CF}) of the ordered configuration spaces of a manifold, and this in particular implies Theorems \ref{thm:MainRational} and \ref{thm:MainClosed} for the field $\bQ$.
\end{rem}

\subsection{Outline}

In \S \ref{sec:SimplicialSpaces} we establish some notation for semi-simplicial spaces. In \S \ref{sec:Construction} we describe what seems to be a new construction in semi-simplicial sets or spaces, and prove a proposition which allows us to identify its homotopy type under certain conditions. Using this we prove the high-connectivity of a semi-simplicial set of ordered subsets of a fixed set $C$ (known as the ``complex of injective words" in the literature).

In \S \ref{sec:StabMaps} we give a careful description of the stabilisation maps $s_\mathcal{E}$ we will be studying, and in \S \ref{sec:SymmetricGroups} we prove Theorem \ref{thm:Main} for $B\Sigma_n = C_n(\bR^\infty)$, as this has a simpler proof than the general case. In \S \ref{sec:ConfigurationSpaces} we give the proof of Theorem \ref{thm:Main} for general configuration spaces.

In \S \ref{sec:Transfer} we recall a transfer argument of Nakaoka which lets us deduce that all the stabilisation maps $s_\mathcal{E}$ are split injective on integral cohomology, and we use this in \S\ref{sec:Improvement} to finish the proof of Theorem \ref{thm:Main} and to prove Theorem \ref{thm:MainRational}.

In \S \ref{sec:Closed} we study the homology of configuration spaces of closed manifolds, and use most of the previous results to prove Theorem \ref{thm:MainClosed}.

\subsection{Acknowledgements}
I would like to thank Martin Palmer for his careful reading of this paper, and for pointing out some errors in earlier versions of the homology stability argument for open manifolds. I am grateful to Thomas Church for sending me a preprint of his work.

\section{Semi-simplicial spaces and resolutions}\label{sec:SimplicialSpaces}

Let $\Delta^{op}$ denote the simplicial category, that is, the opposite of the category $\Delta$ having objects the finite ordered sets $[n] = \{0,1,...,n\}$ and morphisms the weakly monotone maps. A \textit{simplicial object} in a category $\mathcal{C}$ is a functor $X_\bullet : \Delta^{op} \to \mathcal{C}$. Let $\Lambda \subset \Delta$ be the subcategory having all objects but only the strictly monotone maps. Call $\Lambda^{op}$ the \textit{semi-simplicial category} and a functor $X_\bullet : \Lambda^{op} \to \mathcal{C}$ a \textit{semi-simplicial object} in $\mathcal{C}$. A (semi-)simplicial map $f: X_\bullet \to Y_\bullet$ is a natural transformation of functors: in particular, it has components $f_n :X_n \to Y_n$.

The geometric realisation of a semi-simplicial space $X_\bullet$ is
$$\|X_\bullet\| = \coprod_{n \geq 0} X_n \times \Delta^n / \sim$$
where the equivalence relation is $(d_i(x), y) \sim (x, d^i(y))$, for $d^i : \Delta^n \to \Delta^{n+1}$ the inclusion of the $i$-th face. Note that there is a homeomorphism
$$\|X_\bullet\| \cong \hocolim_{\Lambda^{op}} X_\bullet$$
where the homotopy colimit is taken in the category of unpointed topological spaces.

If $X_\bullet$ is a semi-simplicial pointed space, its realisation \textit{as a pointed space} is
$$\|X_\bullet\|_* = \bigvee_{n \geq 0} X_n \rtimes \Delta^n / \sim$$
where $d_i(x) \rtimes y \sim x \rtimes d^i(y)$. Recall that the \textit{half smash product} of a space $Y$ and a pointed space $C$ is the pointed space $C \rtimes Y := C \times Y / * \times Y.$ There is again a homeomorphism
$$\|X_\bullet\|_* \cong \hocolim_{\Lambda^{op}} X_\bullet$$
where the homotopy colimit is taken in the category of pointed topological spaces. If $X_\bullet^+$ denotes levelwise addition of a disjoint basepoint, then there is a homeomorphism $\|X_\bullet^+\|_* \cong \|X_\bullet\|_+$.

\vspace{2ex}

The skeletal filtration of $\|X_\bullet\|$ gives a strongly convergent first quadrant spectral sequence
\begin{equation}\tag{sSS}\label{SSRestrictedSimplicialSpace}
    E^1_{s,t} = h_t(X_s) \Longrightarrow h_{s+t}(\|X_\bullet\|)
\end{equation}
for any connective generalised homology theory $h_*$. The $d^1$ differential is given by the alternating sum of the face maps, $d^1 = \sum (-1)^i (d_i)_*$. This spectral sequence coincides with the Bousfield--Kan spectral sequence for the homology of a homotopy colimit, and is natural for simplicial maps. There is also a pointed analogue, using reduced homology.

\subsection{Pairs of spaces}

If $f: A \to X$ is a continuous map, we write $(X, A)$ for its mapping cone if the map $f$ is clear.

\subsection{Relative semi-simplicial spaces}

Let $f_\bullet: X_\bullet \to Y_\bullet$ be a map of semi-simplicial spaces. Then the levelwise homotopy cofibres form a semi-simplicial pointed space $C_{f_\bullet}$, and
$$\|C_{f_\bullet}\|_* \cong C_{\|f_\bullet\|}$$
as homotopy colimits commute. In particular, the spectral sequence (\ref{SSRestrictedSimplicialSpace}) for this semi-simplicial pointed space is
\begin{equation}\tag{RsSS}\label{SSRelativeRestrictedSimplicialSpace}
    E^1_{s,t} = \tilde{h}_{t}(C_{f_s}) \cong h_t(Y_s, X_s) \Longrightarrow \tilde{h}_{s+t}(C_{\|f_\bullet\|}) \cong h_{s+t}(\|Y_\bullet\|, \|X_\bullet\|).
\end{equation}

\subsection{Augmented semi-simplicial spaces}

An \textit{augmentation} of a (semi-)simplicial space $X_\bullet$ is a space $X_{-1}$ and a map $\epsilon : X_0 \to X_{-1}$ such that $\epsilon d_0 = \epsilon d_1 : X_1 \to X_{-1}$. An augmentation induces a map $\|\epsilon\| : \|X_\bullet\| \to X_{-1}$. In this case there is a spectral sequence defined for $s \geq -1$,
\begin{equation}\tag{AsSS}\label{SSAugmentedRestrictedSimplicialSpace}
E^1_{s,t} = h_t(X_s) \Longrightarrow h_{s+t+1}(C_{\|\epsilon\|}) \cong h_{s+t+1}(X_{-1}, \|X_\bullet\|).
\end{equation}
for any connective generalised homology theory $h_*$. The $d^1$ differentials are as above for $s>0$, and $d^1 : E^1_{0,t} \to E^1_{-1, t}$ is given by $\epsilon_*$.

There is also a relative version of this construction. Let $f: (\epsilon_X : X_\bullet \to X_{-1}) \to (\epsilon_Y : Y_\bullet \to Y_{-1})$ be a map of augmented semi-simplicial spaces. There is a spectral sequence defined for $s \geq -1$,
\begin{equation}\tag{RAsSS}\label{SSRelativeAugmentedRestrictedSimplicialSpace}
E^1_{s,t} = h_t(X_s, Y_s) \Longrightarrow h_{s+t+1}(C_{\|\epsilon_X\|}, C_{\|\epsilon_Y\|}).
\end{equation}

\subsection{Resolutions}

For our purposes, a \textit{resolution} of a space $X$ is an augmented semi-simplicial space $X_\bullet \to X$ such that the map $\|X_\bullet\| \to X$ is a weak homotopy equivalence. An \textit{$n$-resolution} of a space $X$ is an augmented semi-simplicial space $X_\bullet \to X$ such that the map $\|X_\bullet\| \to X$ is $n$-connected.

\begin{comment}
\subsection{Homotopy fibres}\label{sec:ComputingHomotopyFibres}

We will typically show that an augmented semi-simplicial space $X_\bullet \to X$ is an $n$-resolution by showing that the homotopy fibre of $\|X_\bullet\| \to X$ is $(n-1)$-connected. For this it is useful to know that the homotopy fibre may be computed levelwise. Let us write $f_n : X_n \to X$ for the unique composition of face maps.

\begin{lem}
The square
\begin{diagram}
\|\mathrm{hofib}_x f_\bullet \| & \rTo & \|X_\bullet\| \\
\dTo & & \dTo\\
\{ x \} & \rTo & X
\end{diagram}
is weakly homotopy cartesian.
\end{lem}
\begin{proof}
Without loss of generality we may suppose that the $f_n$ are fibrations and that $X$ is a CW complex (so locally contractible), by passing to a homotopy equivalent diagram. As $f_n^{-1}(x) \hookrightarrow \mathrm{hofib}_x f_n$ is a weak homotopy equivalence, it is enough to show that
$$\|f^{-1}_\bullet(x) \| \lra  \|X_\bullet\| \lra X$$
is a homotopy fibre sequence. Certainly $\|f^{-1}_\bullet(x) \|$ is the geometric fibre over $x$, so we will show that $\|X_\bullet\| \to X$ is a local quasi-fibration, hence quasi-fibration, and the result follows.

For any point $a \in X$ let $U_a$ be a contractible neighbourhood, and $b \in U_a$. The fibre over $b$ is $\|f^{-1}_\bullet(b) \|$ and the preimage of $U_a$ is $\|f^{-1}_\bullet(U_a) \|$. As each $f_n$ is a fibration, the map
$$f^{-1}_\bullet(b) \lra f^{-1}_\bullet(U_a) $$
is a levelwise weak homotopy equivalence, so a weak homotopy equivalence on realisation, as required.
\end{proof}
\end{comment}

\section{A construction in semi-simplicial spaces}\label{sec:Construction}

We will describe a simple construction in semi-simplicial spaces, and show how to determine the connectivity of its realisation in certain cases. Though this construction seems very natural, we are not aware of it being discussed in the literature. We refer to \S\ref{sec:SimplicialSpaces} for notation.

Let $X_\bullet$ be a semi-simplicial space, and define a semi-simplicial pointed space $X_{\bullet-1}^+ \rtimes [\bullet]$ having $X_{n-1}^+ \rtimes \{0,...,n\}$ as its space of $n$-simplices, with face maps given by the pointed maps
\begin{equation}
\tilde{d}_i(x \rtimes j) = \begin{cases}
d_i(x) \rtimes (j-1) & \text{if $i<j$,} \\
* & \text{if $i=j$,} \\
d_{i-1}(x) \rtimes j & \text{if $i>j$.}
\end{cases}
\end{equation}

\begin{prop}\label{prop:SimplicialConstruction}
Suppose $X_\bullet$ has the property that the maps
\begin{equation}\tag{$\star$}\label{eq:Resolution}
\|\cdots X_{i+1} \rightrightarrows X_i \| \overset{d_i}\lra X_{i-1}
\end{equation}
are $(k-i)$-connected for all $i$. Then the natural map
$$\Sigma \|X_\bullet\| \lra \|X_{\bullet-1}^+ \rtimes [\bullet]\|_*$$
is $k$-connected. In particular, if the maps (\ref{eq:Resolution}) are equivalences for all $i$ then there is an equivalence $\Sigma \|X_\bullet\| \simeq \|X_{\bullet-1}^+ \rtimes [\bullet]\|_*$.
\end{prop}
\begin{proof}
Note that $\hocofib(\|\cdots X^+_{i+1} \rightrightarrows X^+_i \|_* \overset{d_i}\to X^+_{i-1}) \cong \hocofib(\|\cdots X_{i+1} \rightrightarrows X_i \|_+ \overset{d_i}\to X^+_{i-1}) \simeq \hocofib(\|\cdots X_{i+1} \rightrightarrows X_i \| \overset{d_i}\to X_{i-1})$ which is $(k-i)$-connected by hypothesis.

The semi-simplicial space $X_{\bullet-1}^+ \rtimes [\bullet]$ has a filtration by $F^i (X_{n-1}^+ \rtimes [n]) = \{x \rtimes j \, | \, j \leq i\}$. The inclusion $F^{i-1}(X_{n-1}^+ \rtimes [n]) \to F^{i}(X_{n-1}^+ \rtimes [n])$ is the inclusion of a collection of path components, so a cofibration, and the filtration quotient $F^i (X_{\bullet-1}^+ \rtimes [\bullet]) / F^{i-1} (X_{\bullet-1}^+ \rtimes [\bullet])$ is the semi-simplicial pointed space
\begin{figure}[h]
\centering
\includegraphics[bb = 165 673 485 720, scale=1]{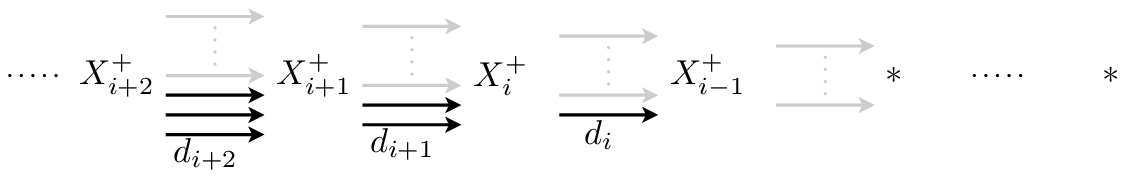}
\end{figure}

\noindent where the grey arrows denote constant maps to the basepoint. The realisation of this semi-simplicial pointed space is
\begin{equation*}
\|F^i (X_{\bullet-1}^+ \rtimes [\bullet]) / F^{i-1} (X_{\bullet-1}^+ \rtimes [\bullet]) \|_* \cong \Sigma^i \hocofib(\|\cdots X^+_{i+1} \rightrightarrows X^+_i \|_* \overset{d_i}\to X^+_{i-1})
\end{equation*}
which is $k$-connected by the remark above. Furthermore, $F^0 \|X_{\bullet-1}^+ \rtimes [\bullet]\|_* \simeq \Sigma \|X_\bullet\|$ by inspection, so the inclusion of the filtration zero part $\Sigma \|X_\bullet\| \to \|X_{\bullet-1}^+ \rtimes [\bullet]\|_*$ is $k$-connected.
\end{proof}

Let us now give our application of this proposition. Let $C$ be a set, and let the semi-simplicial set $F(C)^\bullet$ have $i$-simplices the set of ordered subsets of $C$ of cardinality $(i+1)$. The $j$-th face map is given by forgetting the $j$-th element of the subset.

\begin{prop}\label{prop:ConnectivityCx}
The space $\|F(C)^\bullet\|$ is homotopy equivalent to a wedge of $(\vert C \vert-1)$-spheres.
\end{prop}
\begin{proof}
Certainly $\|F(C)^\bullet\|$ has dimension $(\vert C \vert-1)$, so it is enough to show it is $(\vert C \vert-2)$-connected. We prove this by induction on $\vert C \vert$, and note that it is clear for $\vert C \vert < 2$. Choose a $p \in C$ and consider the inclusion $F(C - p)^\bullet \to F(C)^\bullet$. Note that $\|F(C - p)^\bullet\| \to \|F(C)^\bullet\|$ is nullhomotopic, as it extends over the join with the point $p$, and furthermore $\|F(C - p)^\bullet\| \simeq \vee S^{\vert C \vert-2}$ by inductive assumption.

The levelwise homotopy cofibre of $F(C - p)^\bullet \to F(C)^\bullet$ is $\vee_{j=0}^\bullet F(C-p)^{\bullet-1}_+$ which can be alternatively written as $F(C-p)^{\bullet-1}_+ \rtimes \{0,...,\bullet\}$, which is the construction of this section applied to $F(C-p)^\bullet$. Consider the augmented semi-simplicial space
$$\|\cdots F(C-p)^{i+1} \rightrightarrows F(C-p)^i \| \overset{d_i} \lra F(C-p)^{i-1}.$$
This has fibre $\|F(C-\{p, p_0, ..., p_{i-1}\})^\bullet\|$ over $\{p_0, ..., p_{i-1}\} \in F(C-p)^{i-1}$, which by inductive hypothesis is a wedge of $(\vert C \vert-i-2)$-spheres, so the map $d_i$ is $(\vert C \vert-i-2)$-connected. By Proposition \ref{prop:SimplicialConstruction} the map
$$\vee S^{\vert C \vert-1} \simeq \Sigma \|F(C-p)^\bullet\| \to \|F(C-p)^{\bullet-1}_+ \rtimes \{0,...,\bullet\}\|_*$$
is then $(\vert C \vert-2)$-connected, so $\|F(C-p)^{\bullet-1}_+ \rtimes \{0,...,\bullet\}\|_*$ itself is $(\vert C \vert-2)$-connected. Thus it follows that $\|F(C)^\bullet\|$ is $(\vert C \vert-2)$-connected.
\end{proof}

The method we have described for identifying the connectivity of $\|F(C)^\bullet\|$ is by no means the only way to prove this: another approach is to study the connectivity of the map $\|F(C)^\bullet\| \to \|\bar{F}(C)^\bullet\|$, where $\bar{F}(C)^\bullet$ is the semi-simplicial set with $p$-simplices the set of all $(p+1)$-tuples of elements of $C$. The connectivity of this map can be effectively analysed by studying the link in $\|\bar{F}(C)^\bullet\|$ of ``maximally bad" simplices.

Other proofs of the connectivity of this complex that may be found in the literature under the name of the ``complex of injective words", by Farmer \cite{Farmer}, Bj{\"o}rner and Wachs \cite{BjornerWachs}, and more recently Kerz \cite{Kerz}.

\section{Stabilisation maps}\label{sec:StabMaps}

Let $(X, x_0)$ be a based space. Let $M$ be the interior of a manifold with boundary $\bar{M}$. Let $\bR^d_+$ denote the $d$-dimensional half space, and choose once and for all an embedding $e : \bR^d_+ \to \bR^d_+$ which:
\begin{enumerate}[(i)]
\item is the identity outside of the unit ball,
\item is isotopic to the identity,
\item sends the interior into the interior, and sends the origin into the interior.
\end{enumerate}
For $\mathcal{E} \in \partial \bar{M}$ let $e_\mathcal{E} :  \bar{M} \hookrightarrow \bar{M}$ be a self-embedding which is obtained by identifying some open neighbourhood $U$ of $\mathcal{E}$ with $\bR^d_+$, using $e$ inside this neighbourhood and extending by the identity. Given $C \in C_n(M;X)$ we define
$$s_\mathcal{E}(C) := e_\mathcal{E}(C \cup \{\mathcal{E}\}) \subset C_{n+1}(M;X)$$
where the new point is labeled by $x_0 \in X$. This defines a continuous map
$$s_\mathcal{E} : C_n(M;X) \lra C_{n+1}(M;X).$$
As $e$ is isotopic to the identity, the homotopy class of this map depends only on the path component of $\mathcal{E} \in \partial \bar{M}$.

\section{Symmetric groups}\label{sec:SymmetricGroups}

To warm-up to the proof of Theorem \ref{thm:Main}, we will start with the easier situation where $M=\bR^{\infty}$ and $X=*$. In this case we have $C_n(M;X) := \Emb(\{1,...,n\}, \bR^\infty)/\Sigma_n$, and as $\Emb(\{1,...,n\}, \bR^\infty)$ is a contractible free $\Sigma_n$-space, this is a model for $B\Sigma_n$.

The manifold $\bR^\infty$ is homeomorphic to the interior of the closed $\infty$-disc, so up to homotopy there is a single stabilisation map $s: B\Sigma_{n} \to B\Sigma_{n+1}$, given by adding a point near infinity.

\begin{thm}[Nakaoka \cite{NakaokaDecomposition}]\label{thm:StabSymmetricGroups}
The map $H_*(B\Sigma_{n};\bZ) \to H_*(B\Sigma_{n+1};\bZ)$ is an isomorphism in degrees $2* \leq n$.
\end{thm}

\subsection{Resolution}

Let $E := \{(C, p) \in B\Sigma_n \times \bR^\infty \, | \, p \in C\}$. Forgetting the point $p$ defines a covering space
$$\pi : E \lra B\Sigma_n.$$
Let $B\Sigma_n^i \subset E \times_{B\Sigma_n} E \times_{B\Sigma_n} \cdots \times_{B\Sigma_n} E$ consist of those tuples $(C, p_0, ..., p_i)$ with the $p_j$ all distinct. There are maps $d_j : B\Sigma_n^i \to B\Sigma_n^{i-1}$ for $j=0,...,i$ given by forgetting the $j$-th point. Furthermore, there is a unique map $B\Sigma_n^i \to B\Sigma_n$ given by forgetting all points.

\begin{prop}\label{prop:ResolutionSymmetricGroups}
$B\Sigma_n^\bullet \to B\Sigma_n$ is an augmented semi-simplicial space. The fibre of $\|B\Sigma_n^\bullet\| \to B\Sigma_n$ over a configuration $C \in B\Sigma_n$ is homotopy equivalent to a wedge of $(n-1)$-spheres. In particular this gives a $(n-1)$-resolution of $B\Sigma_n$.
\end{prop}
\begin{proof}
That $B\Sigma_n^\bullet \to B\Sigma_n$ is an augmented semi-simplicial space is clear: all that is required is to observe that the maps $d_j$ respect the simplicial identities. The fibre over a configuration $C$ is nothing but the realisation of the semi-simplicial space $F(C)^\bullet$, which we have shown in Proposition \ref{prop:ConnectivityCx} to be a wedge of $(n-1)$-spheres.
\end{proof}

\begin{prop}
There are homotopy equivalences $B\Sigma_n^i \simeq B\Sigma_{n-i-1}$. Under this identification the face maps $d_j : B\Sigma_{n}^i \to B\Sigma_{n}^{i-1}$ and the augmentation map $B\Sigma_n^0 \to B\Sigma_n$ are all homotopic to the stabilisation map.
\end{prop}
\begin{proof}
Note that $\Emb(\{1, ..., n-i-1\}, \bR^\infty \setminus \{ 1, ..., i+1\})$ is still contractible, and a free $\Sigma_{n-i-1}$-space. The quotient by $\Sigma_{n-i-1}$ is precisely a fibre of the fibration
$$B\Sigma_n^i \overset{\pi}\lra \Emb(\{1,...,i+1\}, \bR^\infty) \simeq *,$$
where the map $\pi$ discards the undistinguished points. This establishes the homotopy equivalence $B\Sigma_n^i \simeq B\Sigma_{n-i-1}$. The face maps all differ by some permutation of the distinguished points, $\sigma : B\Sigma_n^i \to B\Sigma_n^i$, so to show the face maps are homotopic it is enough to show these permutation maps are homotopic to the identity.

Note $B\Sigma_n^i$ is homotopy equivalent to its subspace $Y$ where the distinguished points have first coordinate 1, and the undistinguished points have negative first coordinate. There is a homeomorphism
$$Y \cong \Emb(\{1, ..., n-i-1\}, (-\infty,0)\times \bR^\infty)/\Sigma_{n-i-1} \times \Emb(\{n-i, ..., n\}, \{1\} \times \bR^\infty)$$
and the automorphism $\sigma$ preserves the subspace $Y$, acts trivially on the first factor and by a permutation on the second. The second factor is contractible, so $\sigma$ is homotopic to the identity on $Y$, and hence also on $B\Sigma_n^i$.
\end{proof}

\begin{rem}
There is an alternative --- though equivalent --- point of view on this resolution. Let us \textit{define} $B\Sigma_n := * \hcoker \Sigma_n$ to be the homotopy quotient. The semi-simplicial set $F(\{1, ..., n\})^\bullet$ has an action of $\Sigma_n$, and we can \textit{define} $B\Sigma_n^i := F(\{1, ..., n\})^i \hcoker \Sigma_n$ to be the homotopy quotient.

Then there is a homotopy equivalence $\|F(\{1, ..., n\})^\bullet\| \hcoker \Sigma_n \simeq \|B\Sigma_n^{\bullet}\|$, so a fibration sequence
$$\|F(\{1, ..., n\})^\bullet\| \lra \|B\Sigma_n^{\bullet}\| \lra B\Sigma_n.$$
It then remains to identify $B\Sigma_n^i$, which may be seen to be $B\Sigma_{n-i-1}$ as $\Sigma_n$ acts transitively on $F(\{1, ..., n\})^i$ with stabiliser $\Sigma_{n-i-1}$.
\end{rem}

\subsection{Proof of Theorem \ref{thm:StabSymmetricGroups}}

Note that Theorem \ref{thm:StabSymmetricGroups} is trivially true for $n=0$, as the spaces $B\Sigma_n$ are all connected. Consider the spectral sequence (\ref{SSRestrictedSimplicialSpace}) in integral homology applied to the semi-simplicial space $B\Sigma_{n+1}^\bullet$. It is
$$E^1_{s,t} = H_t(B\Sigma_{n-s}) \Longrightarrow H_{s+t}(\|B\Sigma^\bullet_{n+1}\|) \cong H_{s+t}(B\Sigma_{n+1}) \quad \text{for} \quad s+t \leq n-1$$
and we wish to deduce something about the range $2* \leq n$, so need $\lfloor\frac{n}{2}\rfloor \leq n-1$, which holds for all $n \geq 1$. In this spectral sequence the $d^1$ differential is given by the alternating sum of the face maps. However, these are all freely homotopic, so induce the same map on homology. Thus $d^1 : E^1_{odd, *} \to E^1_{even, *}$ is zero and $d^1 : E^1_{even, *} \to E^1_{odd, *}$ is given by the stabilisation map,
$$s_* : H_t(B\Sigma_{n-2s}) \lra H_t(B\Sigma_{n-2s+1}).$$
Applying Theorem \ref{thm:StabSymmetricGroups} for $n' < n$ points, we see that this stabilisation map is an isomorphism for $2t \leq n-2s$, so that $E^2_{*,*}$ is trivial for bidegrees $(2s,t)$ with $2(t+2s) \leq n+2s$ and $s>0$, and bidegrees $(2s+1, t)$ with $2(t+2s+1) \leq n+2s+2$ and $s>0$.
\begin{figure}[h]
\centering
\includegraphics[bb = 0 0 139 120]{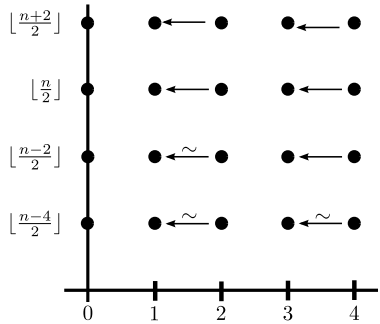}
\caption{$E^1$ page of the spectral sequence converging to $H_*(|B\Sigma_{n+1}^\bullet|)$.}
\label{fig:E1SymmetricGroups}
\end{figure}
Observing Figure \ref{fig:E1SymmetricGroups}, in total degrees $2* \leq n$ the spectral sequence collapses at $E^2$ and is concentrated along $s=0$. Thus Theorem \ref{thm:StabSymmetricGroups} holds.

\section{Configuration spaces}\label{sec:ConfigurationSpaces}

%For $Y$ a connected manifold and $X$ a path connected space, let us define
%$$C_n(Y ; X) := \Emb(\{1,...,n\}, Y) \times_{\Sigma_n} X^n$$
%to be the moduli space of $n$ unordered points in $Y$ labelled by $X$. 

The spectral sequence argument we will use is different to that in the last section, and is closer to that used in \cite{R-WResolution} to prove homological stability for moduli spaces of surfaces with tangential structure.

Let us write $M_k$ for the manifold $M$ with $k$ points removed (as $M$ is assumed to be connected, the diffeomorphism type of $M_k$ is independent of which points are removed). Define spaces
$$C_n(M; X)^i := \{(C, p_0,...,p_i) \in C_n(M ;X) \times M^{i+1} \,\, | \,\, p_j \in C, \, p_j \neq p_k \}.$$
There are maps $d_j : C_n(M; X)^i \to C_n(M; X)^{i-1}$ given by forgetting the $j$-th point, and $C_n(M; X)^i \to C_n(M; X)$ given by forgetting all points.

\begin{prop}\label{prop:ConfigurationSpaceResolution}
$C_n(M ;X)^\bullet \to C_n(M ;X)$ is an augmented semi-simplicial space, and an $(n-1)$-resolution. Furthermore, there are homotopy fibre sequences
$$C_{n-i-1}(M_{i+1} ;X) \lra C_n(M ;X)^i \overset{\pi}\lra \Emb(\{1,...,i+1\}, M) \times X^{i+1} = \widetilde{C}_{i+1}(M;X).$$
\end{prop}
\begin{proof}
The first statement requiring proof is that the fibration $\|C_n(M ;X)^\bullet\| \to C_n(M ;X)$ is $(n-1)$-connected. However, its fibre over a labeled configuration $C$ in $M$ is $\| F(C)^\bullet\|$, which we showed in Proposition \ref{prop:ConnectivityCx} to be a wedge of $(n-1)$-spheres.

The map $\pi$ sends a tuple $(C, p_0,...,p_i)$ to $(p_0,...,p_i , \ell(p_0),..., \ell(p_i))$ where $\ell(p_j)$ denotes the label in $X$ of $p_j \in C$. This is a fibration (in fact a fibre bundle) and the fibre over $(p_0,...,p_i, x_0,...,x_i)$ is $C_{n-i-1}(M \setminus \{p_0,...,p_i \} ; X)$.
\end{proof}

Consider the stabilisation map $s_\mathcal{E} : C_{n}(M ;X) \to C_{n+1}(M ;X)$ for $\mathcal{E} \in \partial \bar{M}$. Let us write $R_n(M,\mathcal{E})$ for the pair $(C_{n+1}(M;X), C_n(M;X))$, where $C_n(M;X)$ is thought of as a subspace via the map $s_\mathcal{E}$. If the reader prefers they can consider $R_n(M,\mathcal{E})$ to be the mapping cone of $s_\mathcal{E}$: we will be interested only in its connectivity. The map $s_\mathcal{E}$ induces a simplicial map of resolutions $s_\mathcal{E}^\bullet : C_n(M;X)^\bullet \to C_{n+1}(M;X)^\bullet$ and we write $R_n(M,\mathcal{E})^\bullet$ for the simplicial pair $(C_{n+1}(M;X)^\bullet, C_n(M;X)^\bullet)$. There is then an augmented semi-simplicial object in the category of pairs of spaces
$$R_n(M,\mathcal{E})^\bullet \lra R_n(M,\mathcal{E}),$$
and this is an $n$-resolution.

We will use this resolution to prove the following theorem, which is a slightly weaker result than Theorem \ref{thm:Main}. We will explain how to obtain the full strength of Theorem \ref{thm:Main} in the following two sections.

\begin{thm}\label{thm:ConfSpaceBasic}
The pair $R_n(M, \mathcal{E})$ has trivial homology in degrees $2* \leq n$. Equivalently, $s_\mathcal{E} : C_n(M;X) \to C_{n+1}(M;X)$ is a homology epimorphism in degrees $2* \leq n$ and a homology isomorphism in degrees $2* \leq n-2$.
\end{thm}

\subsection{Proof of Theorem \ref{thm:ConfSpaceBasic}}
We proceed by induction on $n$. As $M$ is connected and of dimension at least two, all the spaces involved are connected and so the statement is trivially true for $n \leq 1$. For $n \geq 2$, we can apply (\ref{SSRelativeAugmentedRestrictedSimplicialSpace}) to the augmented semi-simplicial pair of spaces $R_n(M,\mathcal{E})^\bullet \lra R_n(M,\mathcal{E}).$ It has the form
$$E^1_{s,t} = H_t(C_{n+1}(M ;X)^s, C_{n}(M ;X)^s) \Longrightarrow 0 \text{\quad for $s+t \leq n-1$}$$
and $n-1 \geq n/2$ for $n \geq 2$, so the range we wish to study is within the range that the spectral sequence converges to zero.

By Proposition \ref{prop:ConfigurationSpaceResolution} there is a relative Serre spectral sequence
\begin{eqnarray*}
\tilde{E}^2_{s,t} = H_t(\widetilde{C}_{i+1}(M;X) ; H_s(C_{n-i}(M_{i+1} ; X), C_{n-i-1}(M_{i+1} ; X))) \\ \nonumber\Longrightarrow H_{s+t}(C_{n+1}(M ; X)^i, C_{n}(M ; X)^i),
\end{eqnarray*}
and applying Theorem \ref{thm:Main} to $M_{i+1}$ and $(n-i-1)$ points we see that $\tilde{E}^2_{s,t}=0$ for $2s \leq n-i-1$. Thus $H_{*}(C_{n+1}(M ; X)^i, C_{n}(M ; X)^i)$ is trivial in degrees $2* \leq n-i-1$, and the inclusion of the fibre to the total space gives a homology epimorphism in degrees $2* \leq n-i+1$.

This implies that $E^1_{s,t}=0$ for $2t \leq n-s-1$, so for $2(s+t) \leq n+s-1$. Thus the augmentation
$$d_1 : H_*(C_{n+1}(M ;X)^0, C_{n}(M ;X)^0) \lra H_*(C_{n+1}(M ;X), C_{n}(M ;X))$$
is an epimorphism in degrees $2* \leq n$. By the Serre spectral sequence calculation above, the inclusion of the relative fibre
$$(C_{n}(M_1 ;X), C_{n-1}(M_1 ;X)) \lra (C_{n+1}(M ;X)^0, C_{n}(M ;X)^0)$$
over a point $(1, t) \in \widetilde{C}_1(M;X)$ gives a homology epimorphism in degrees $2* \leq n+1$. 

The space $C_n(M_1;X)$ consists of configurations of fewer points, but on the punctured manifold $M_1$ instead of $M$. We need to relate this to $C_n(M;X)$, so it is necessary to study the map $p_{\mathcal{E}'}: C_n(M;X) \to C_n(M_1;X)$ that adds a puncture to the manifold $M$ near the point $\mathcal{E}'$, and its natural partner, the map $f_1 : C_n(M_1;X) \to C_n(M;X)$ that fills in a puncture. These are defined identically to the stabilisation maps in \S \ref{sec:StabMaps}. The composition
$$C_n(M;X) \overset{p_{\mathcal{E}'}}\lra C_n(M_1;X) \overset{f_1}\lra C_n(M;X)$$
is homotopic to the identity, so studying the homological effect of one map is equivalent to studying the homological effect of the other.

These maps are \emph{not} homology equivalences, but they are in the relative situation. To be precise, let $\mathcal{E}$ and $\mathcal{E}'$ be distinct, and choose neighbourhoods of them that are disjoint to define the maps $s_\mathcal{E}$ and $p_{\mathcal{E}'}$. Then these maps commute, and so $p_{\mathcal{E}'}$ defines a relative map $p_{\mathcal{E}'}: R_n(M,\mathcal{E}) \to R_n(M_1,\mathcal{E})$.%, the maps $R_n(M,\mathcal{E}) \to R_n(M_1,\mathcal{E}) \to R_n(M,\mathcal{E})$ are both homology equivalences in a range.

\begin{prop}\label{prop:Punctures}
Suppose Theorem \ref{thm:ConfSpaceBasic} holds up to $(n-1)$. The relative homology groups $H_*(C_{n+1}(M;X), C_n(M;X))$ exhibit homological stability for adding and removing punctures; more precisely, the maps $R_n(M,\mathcal{E}) \overset{p_{\mathcal{E}'}}\to R_n(M_1,\mathcal{E}) \overset{f_1}\to R_n(M,\mathcal{E})$ are both homology equivalences in degrees $2* \leq n+1$.
\end{prop}
\begin{proof}
Let $D$ be the closed $\mathrm{dim}(M)$-dimensional disc of unit radius, and choose an open embedding $e: D \hookrightarrow M$. We will decompose the space $C_n(M;X)$ into a pair of open sets: let $U \subset C_n(M;X)$ be the subspace consisting of those configurations with a unique closest point in $D$ to $0$; let $V \subset C_n(M;X)$ be the subspace consisting of those configurations with no point in $D$ at $0$. The sets $U$ and $V$ give an open cover of $C_n(M;X)$. We identify the homotopy types of $U$, $V$ and $U \cap V$ as follows.

\begin{enumerate}[(i)]
\item There is a fibre sequence $C_{n-1}(M_1;X) \to U \to D \times X$, where the second map picks out the unique closest point in $D$ to 0 and its label, and this fibration is trivial.
\item There is a homotopy equivalence $V \cong C_n(M \setminus 0;X) \simeq C_n(M_1;X)$.
\item There is a fibre sequence $C_{n-1}(M_1;X) \to U \cap V \to (D \setminus 0) \times X$, which is the restriction of the fibration in (i) and hence trivial.
\end{enumerate}

By excision, the homology of the pair
$$(U \cup V, V) \simeq (C_n(M;X), C_n(M_1;X))$$
is canonically isomorphic to that of the pair
$$(U, U \cap V) \simeq C_{n-1}(M_1;X)_+ \wedge X_+ \wedge (D, D\setminus 0).$$
Then, as the homotopy cofibre of $f_1 : R_n(M_1,\mathcal{E}) \lra R_n(M,\mathcal{E})$ can be identified with that of
$$s_\mathcal{E} : (C_n(M;X), C_n(M_1;X)) \lra (C_{n+1}(M;X), C_{n+1}(M_1;X)),$$
its homology can be identified with that of the homotopy cofibre of
$$C_{n-1}(M_1;X)_+ \wedge X_+ \wedge S^{\mathrm{dim}(M)} \overset{s_\mathcal{E} \wedge \mathrm{Id}}\lra C_n(M_1;X)_+ \wedge X_+ \wedge S^{\mathrm{dim}(M)}$$
which is $\lfloor\tfrac{n-1}{2}\rfloor + \mathrm{dim}(M) \geq \lfloor\tfrac{n+1}{2}\rfloor$ connected, by applying Theorem \ref{thm:ConfSpaceBasic} for $(n-1)$.
\end{proof}

\begin{rem}
A corollary to the proof of the above proposition is the existence of the homotopy cofibre sequence
\begin{equation}\label{eq:CofSeq}
C_n(M \setminus *) \lra C_n(M) \lra S^d \wedge C_{n-1}(M \setminus *)_+.
\end{equation}
We will come back to this sequence in \S \ref{sec:Closed}.
\end{rem}

Let $\mathcal{E}' \in \partial \bar{M}$ lie in the same path component of the boundary as $\mathcal{E}$. By this proposition, the map
$$p_\mathcal{E'} : (C_{n}(M ;X), C_{n-1}(M ;X)) \lra (C_{n}(M_1 ;X), C_{n-1}(M_1 ;X))$$
which adds a puncture near $\mathcal{E}'$ also induces a homology epimorphism in degrees $2* \leq n$. Thus the composition
\begin{eqnarray*}
(C_{n}(M ;X), C_{n-1}(M ;X)) \overset{p_{\mathcal{E}'}}\lra (C_{n}(M_1 ;X), C_{n-1}(M_1 ;X)) \lra\\
(C_{n+1}(M ;X)^0, C_{n}(M ;X)^0) \overset{\epsilon}\lra (C_{n+1}(M ;X), C_{n}(M ;X))    
\end{eqnarray*}
is a homology epimorphism in degrees $2* \leq n$. We will now show that the map is also trivial in this range of degrees, which will imply Theorem \ref{thm:Main}. 

This composition corresponds to the commutative square
\begin{diagram}
C_{n-1}(M ;X) & \rTo^{s_\mathcal{E}} & C_{n}(M ;X)\\
\dTo^{s_{\mathcal{E}'}} & \ldTo[dotted]^{\mathrm{Id}} & \dTo^{s_{\mathcal{E}'}}\\
C_{n}(M ;X) & \rTo^{s_\mathcal{E}} & C_{n+1}(M ;X).
\end{diagram}
Choosing a path $\mathcal{E} \leadsto \mathcal{E}'$ in $\partial\bar{M}$ gives a homotopy $H$ making the bottom triangle commute, and a homotopy $G$ making the top triangle commute. We would like to assert that this determines a relative nullhomotopy from the top pair of spaces to the bottom pair, but this in not true: it merely gives a factorisation up to homotopy
$$(C_{n}(M ;X), C_{n-1}(M ;X)) \overset{\partial}\lra (I, \partial I) \wedge C_{n-1}(M;X)_+ \overset{\tau}\lra (C_{n+1}(M ;X), C_{n}(M ;X))$$
where $\partial$ is the connecting map in the Puppe sequence, and $\tau$ is the map determined by the self-homotopy of the map $s_\mathcal{E} \circ s_{\mathcal{E}'}$ given by
$$s_\mathcal{E} \circ s_{\mathcal{E}'} \overset{G}\leadsto s_\mathcal{E} \circ s_{\mathcal{E}} \overset{H}\leadsto s_{\mathcal{E}'} \circ s_{\mathcal{E}} = s_\mathcal{E} \circ s_{\mathcal{E}'}.$$
Considering the homotopies $H$ and $G$, we see that the effect of $\tau$ on homology is determined by the map
$$T_\mathcal{E} : S^1 \times C_{n-1}(M;X) \overset{t \times \mathrm{Id}}\lra C_2(\bR^d) \times C_{n-1}(M;X) \overset{\text{glue}_\mathcal{E}}\lra C_{n+1}(M;X)$$
evaluated on the fundamental class of $S^1$, where $t$ is the loop that interchanges the two configuration points and the gluing map adds two new points near $\mathcal{E}$.

\begin{lem}\label{lem:TauIsZero}
Suppose Theorem \ref{thm:ConfSpaceBasic} holds up to $(n-2)$. The map $\tau$ is trivial on homology in degrees $2* \leq  n$. If $d > 2$ it is trivial on homology with $\bZ[\tfrac{1}{2}]$-module coefficients in all degrees.
\end{lem}
\begin{proof}
For the statement about $\bZ[\tfrac{1}{2}]$-module coefficients, note that
$$t_*[S^1] \in H_1(C_2(\bR^d);\bZ)$$
which is $\bZ/2$ as long as $d > 2$, and so is trivial with any $\bZ[\tfrac{1}{2}]$-module coefficients. In general, we have a commutative diagram
\begin{diagram}
S^1 \times C_{n-1}(M ;X) & \rTo^{T_\mathcal{E}} & C_{n+1}(M ;X) & \rTo & (C_{n+1}(M ;X), C_{n}(M ;X))\\
\uTo^{s_{\mathcal{E}'}} &  & \uTo^{s_{\mathcal{E}'}} & \ruTo_*\\
S^1 \times C_{n-2}(M ;X) & \rTo^{T_\mathcal{E}} & C_{n}(M ;X)
\end{diagram}
as long as $n \geq 2$. The left hand map is a homology surjection in degrees $2(*-1) \leq n-2$ by assumption.
\end{proof}

Thus the map
$$(C_{n}(M ;X), C_{n-1}(M ;X)) \lra (C_{n+1}(M ;X), C_{n}(M ;X))$$
is both a homology surjection and zero in degrees $2* \leq n$, and hence $R_n(M, \mathcal{E})$ has trivial homology in degrees $2* \leq n$.

\section{The transfer and decompositions}\label{sec:Transfer}

There is a technique implicit in the work of Nakaoka \cite{NakaokaDecomposition} and axiomatised by Dold \cite{DoldDecomposition}, which is extremely useful in the study of configuration spaces, and more generally for spaces constructed out of free symmetric group actions. It will allow us to deduce the full statement of Theorem \ref{thm:Main} from Theorem \ref{thm:ConfSpaceBasic}.

Let $C_{n,m}(M)$ denote the space of configurations of $n$ points in $M$, grouped into sets of size $m$ and $n-m$ respectively. Forgetting the groupings gives a covering map
$$C_{n,m}(M) \lra C_{n}(M)$$
of degree $\binom{n}{m}$, and we denote by $t_{n, m}$ the map on homology
$$t_{n,m} : H_*(C_n(M);\bZ) \overset{\trf}\lra H_*(C_{n, m}(M);\bZ) \overset{\text{forget}}\lra H_*(C_{m}(M);\bZ)$$
constructed using the transfer map for this covering. We denote by $t_n$ the map $t_{n, n-1}$, and by $i_n$ the stabilisation map
$$H_*(C_n(M);\bZ) \overset{s_\mathcal{E}}\lra H_*(C_{n+1}(M);\bZ).$$
The definition of the transfer gives the equation
$$t_n \circ i_{n-1} = i_{n-2} \circ t_{n-1} + \mathrm{Id}$$
and more generally
$$t_{n, m} \circ i_{n-1} = i_{n-2} \circ t_{n-1, m-1} + t_{n-1, m}.$$
Furthermore, we have that $t_{m+1} \circ \cdots \circ t_{n} = (n-m)! \cdot t_{n, m}$, and this determines the relations
$$\binom{n-k}{n-m}\cdot t_{m, k} \circ t_{n,m} = t_{n, k}$$
between the transfer maps. If we denote by $A_n$ the group $H_i(C_n(M);\bZ)$ then we find ourselves in the situation described by Dold \cite[Lemma 2]{DoldDecomposition}, and writing
$$B_n := \Coker(i_{n-1} : H_i(C_{n-1}(M);\bZ) \to H_i(C_{n}(M);\bZ))$$
Dold shows inductively that
$$\mathrm{Id} \oplus t_{n, n-1} \oplus t_{n, n-2} \oplus \cdots \oplus t_{n, n} : A_n \lra B_n \oplus B_{n-1} \oplus B_{n-2} \oplus \cdots \oplus B_{0}$$
is an isomorphism, and hence that the $i_n$ are all split injective. Furthermore, he shows that rationally the endomorphism $t_{n+1} \circ i_n$ of $A_n$ preserves this decomposition, and is multiplication by a non-zero scalar on each summand; hence it is an isomorphism. This establishes

\begin{prop}\label{prop:Trf}
The maps $s_\mathcal{E}$ are split injective in integral homology in all degrees, and the map $t_n : H_*(C_n(M);\bQ) \to H_*(C_{n-1}(M);\bQ)$ is an isomorphism in the range $i_{n-1}$ is.
\end{prop}
%\begin{proof}
%That the $s_\mathcal{E}$ are split injective is immediate from the result of Dold. Hence Theorem \ref{thm:Inverting2} shows that $i_{n-1}$ is an rational isomorphism in degrees $* \leq n-1$. Now $t_n \circ i_{n-1}$ is a rational isomorphism, and so $t_n$ is also an isomorphism in degrees $* \leq n-1$.
%\end{proof}

%This proposition establishes the final form of Theorem \ref{thm:Main}.

\begin{proof}[Proof of Theorem \ref{thm:Main}]
By the first part of Proposition \ref{prop:Trf} the maps $s_\mathcal{E}$ are split injective in integral homology in all degrees. By Theorem \ref{thm:ConfSpaceBasic} the maps $s_\mathcal{E} : C_n(M;X) \to C_{n+1}(M;X)$ are integral homology surjections in degrees $2* \leq n$. Hence they are isomorphisms in this range of degrees.
\end{proof}

\section{An improvement}\label{sec:Improvement}

Let us briefly describe how the stability range may be improved after we make certain assumptions. In particular, let us assume that $\dim(M) >2$ and that we take rational coefficients. The following is a restatement of Theorem \ref{thm:MainRational} from the introduction.

\begin{thm}\label{thm:Inverting2}
Let $M$ be a connected open manifold which is the interior of a manifold with boundary and has dimension greater than 2. Then each map $s_\mathcal{E} : C_n(M;X) \to C_{n+1}(M;X)$ is a rational homology isomorphism in degrees $* \leq n$.
\end{thm}

Given the injectivity of stabilisation maps, the statement is equivalent to showing that $R_n(M,\mathcal{E})$ has trivial rational homology in degrees $* \leq n$. This is certainly true for $n=0$ (as $C_0(M;X) = *$). For the inductive step, we proceed as in the proof of Theorem \ref{thm:Main}. The $E^1$-page of the spectral sequence now has the following form, and converges to zero in total degree $\leq n-1$.
\begin{figure}[h]
\centering
\includegraphics[bb = 0 0 176 147]{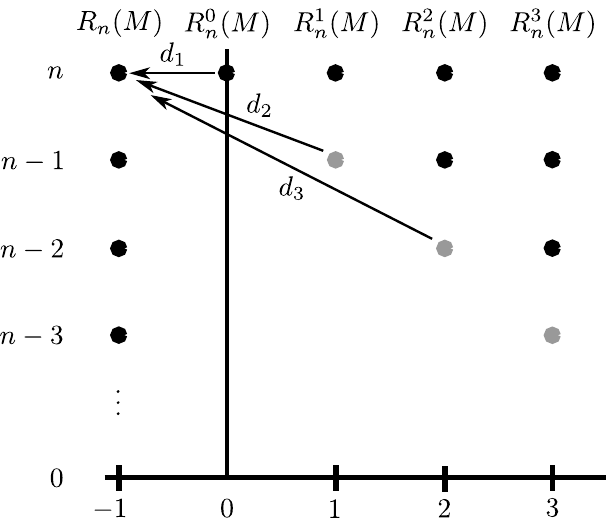}
%\caption{$E^1$ page of the spectral sequence converging to $H_*(|B\Sigma_{n+1}^\bullet|)$.}
\label{fig:E1}
\end{figure}

\begin{prop}
The indicated differentials $d_2$, $d_3, \ldots, d_{n}$ are zero.
\end{prop}
\begin{proof}
Let $1 \leq k \leq n-1$ and consider the map
$$C_{n-k}(M;X) \times C_{k+1}(\bR^d;X) \lra C_{n+1}(M;X),$$
where $d = \dim(M)$, which glues on near $\mathcal{E}$ a disk with $(k+1)$ configuration points in it. The subspace
$$s_\mathcal{E}(C_{n-k-1}(M;X)) \times C_{k+1}(\bR^d;X) \cup C_{n-k}(M;X) \times s_\mathcal{E}(C_k(\bR^d;X))$$
may be compressed into $s_\mathcal{E}(C_n(M;X))$, which induces
$$R_{n-k-1}(M, \mathcal{E}) \wedge R_{k}(\bR^d) \lra R_n(M, \mathcal{E}).$$
This extends to a map of resolutions $R_{n-k-1}(M, \mathcal{E}) \wedge R^\bullet_{k}(\bR^d) \to R^\bullet_n(M, \mathcal{E})$, which gives a map $\bar{E}^r_{s,t} \to E^r_{s,t}$ of augmented spectral sequences. Just as for $E^r_{s,t}$, by the inductive hypothesis we may compute that $\bar{E}^1_{s,t}$ is zero for $t+s \leq n-1$ and $s \geq -1$.
On $k$-simplices we have the map
$$R_{n-k-1}(M, \mathcal{E}) \wedge \widetilde{C}_{k+1}(\bR^d)_+ \lra R_n^{k}(M, \mathcal{E}),$$
which is surjective in degrees $* \leq n-k$ by the analogue of Proposition \ref{prop:Punctures}. The group $\bar{E}^1_{k, n-k}$ consists of permanent cycles, so on the associated map of spectral sequences we have the commutative square
\begin{diagram}
H_{n-k}(R_n^k(M, \mathcal{E})) & \rTo^{d_{k+1}} & H_n(R_n(M, \mathcal{E}))\\
\uOnto & & \uTo\\
H_{n-k}(R_{n-k-1}(M, \mathcal{E}) \wedge \widetilde{C}_{k+1}(\bR^d;X)_+) & \rTo^{d_{k+1}} & H_n(R_{n-k-1}(M, \mathcal{E}) \wedge R_{k}(\bR^d))=0
\end{diagram}
and so the image of the differential $d_{k+1}$ is zero, as required.
\end{proof}

\begin{lem}
Let $N \subset M$ be a nonseparating submanifold of positive codimension containing $\mathcal{E}$. Then the map $R_n(N) \to R_n(M)$ is zero on homology with $\bZ[\tfrac{1}{n+1}]$-module coefficients.
\end{lem}
\begin{proof}
As stabilisation maps are split injective, we have a map of short exact sequences
\begin{diagram}
0 & \rTo & H_*(C_n(N;X)) & \rTo^{s_\mathcal{E}} & H_*(C_{n+1}(N;X)) & \rTo & H_*(R_n(N)) & \rTo & 0\\
 & & \dTo & & \dTo^i & & \dTo\\
0 & \rTo & H_*(C_n(M;X)) & \rTo^{s_\mathcal{E}} & H_*(C_{n+1}(M;X)) & \rTo & H_*(R_n(M)) & \rTo & 0.
\end{diagram}
We also have a commutative square
\begin{diagram}
H_*(C_{n+1, 1}(N;X)) & \rTo^\pi & H_*(C_{n+1}(N;X))\\
\dTo^{forget} & & \dTo^{i}\\
H_*(C_{n}(M;X)) & \rTo^{s_\mathcal{E}} & H_*(C_{n+1}(M;X))
\end{diagram}
where $\pi : C_{n+1, 1}(N;X) \to C_{n+1}(N;X)$ is the $(n+1)$-fold covering map. This square commutes as the forgotten point may be connected to $\mathcal{E}$ inside the complement of $N$, which gives a homotopy between the two compositions. As the map $\pi$ is surjective with $\bZ[\tfrac{1}{n+1}]$-module coefficients, it follows that the image of $i$ is contained in the image of $s_\mathcal{E}$ and hence that the right hand map in the first diagram is zero.
\end{proof}

\begin{prop}
The differential $d_{n+1} : H_0(R^n_n(M)) \to H_n(R_n(M))$ is zero with $\bZ[\tfrac{1}{n+1}]$-module coefficients.
\end{prop}
\begin{proof}
Consider a trivially embedded 2-disc $D^2 \subset M$ which contains $\mathcal{E}$. We obtain a map $R_n(D^2, \mathcal{E}) \to R_n(M, \mathcal{E})$ and an associated map of resolutions $R^\bullet_n(D^2, \mathcal{E}) \to R^\bullet_n(M, \mathcal{E})$. On $n$-simplices this is
$$(\widetilde{C}_{n+1}(D^2), \emptyset) \lra (\widetilde{C}_{n+1}(M), \emptyset)$$
and in particular is surjective on zeroth homology. Comparing the spectral sequences for the two manifolds we see the same vanishing range for $D^2$ as for $M$, and we obtain the commutative square
\begin{diagram}
H_{0}(R_n^n(M, \mathcal{E})) & = & H_0(\widetilde{C}_{n+1}(M)) & \rTo^{d_{n+1}} & H_n(R_n(M, \mathcal{E}))\\
\uOnto & &  & & \uTo^0\\
H_{0}(R_n^n(D^2, \mathcal{E})) & = & H_0(\widetilde{C}_{n+1}(D^2))& \rTo^{d_{n+1}} & H_n(R_n(D^2, \mathcal{E}))
\end{diagram}
for the differential on the $(n+1)$-st page. Hence the differential is zero for $M$.
\end{proof}

These two propositions imply that $\epsilon: R_n^0(M, \mathcal{E}) \to R_n(M, \mathcal{E})$ is surjective on homology in degrees $* \leq n$. The analogue of Proposition \ref{prop:Punctures} shows that $R_{n-1}(M) \to R_n^0(M)$ is a homology surjection in degrees $* \leq n+1$, and so
$$R_{n-1}(M) \lra R_n^0(M) \overset{\epsilon}\lra R_n(M)$$
is a homology surjection in degrees $* \leq n$, but by Lemma \ref{lem:TauIsZero} this map is zero with $\bZ[\tfrac{1}{2}]$-module coefficients when $M$ has dimension $> 2$.

\section{Closed manifolds}\label{sec:Closed}

It has been known for some time that the homology of configuration spaces of closed manifolds does not stabilise. Certainly there are no natural stabilisation maps between them that could induce stabilisation. In fact, one may compute
$$H_1(C_n(S^2);\bZ) \cong \bZ/(2n-2)$$
from the presentation of the spherical braid group given in \cite{FvB}
and in a sense this example is representative of all that can go wrong. In the remainder of this section, we assume that $M$ has finite-dimensional $\bF$-homology in each degree, so that the same is true of $C_n(M)$.

\begin{thm}\label{thm:StableHomologyClosed}
Suppose that $d$ is odd and $\bF$ is any field, or that $d$ is even and $\bF$ is $\bF_2$ or $\bQ$. Then $\mathrm{dim}_\bF H_*(C_n(M);\bF)$ is independent of $n$ in the range that $H_*(C_n(M\setminus *);\bF)$ is stable.
\end{thm}

The configuration space $C_n(M)$ of a closed manifold has a natural length two filtration
$$F_0 = C_n(M \setminus *) \subset F_1 = C_n(M).$$
The filtration quotient $F_1 / F_0$ is equivalent to $S^d \wedge C_{n-1}(M \setminus *)_+$ by (\ref{eq:CofSeq}), and so $F_0$ and $F_1 / F_0$ \emph{do} have homology stability. We can recover the suspension $\Sigma C_n(M)_+$ as the homotopy cofibre of the connecting map in the Puppe sequence,
$$F_1 / F_0 \overset{\Delta_n}\lra \Sigma {F_0}_+ \lra \Sigma C_n(M)_+,$$
and as the filtration quotients stabilise we are lead to studying the connecting map $\Delta_n : F_1 / F_0 \to \Sigma {F_0}_+$ and its effect on homology.

In the cofibre sequence
\begin{equation}\label{eq:CofSeqClosed}
C_n(M) \lra S^{d} \wedge C_{n-1}(M\setminus *)_+ \overset{\Delta_n}\lra \Sigma C_n(M \setminus *)_+ \dashrightarrow \Sigma C_n(M)_+,
\end{equation}
the effect of the map $\Delta_n$ on homology is determined by the map
$$p_n : S^{d - 1} \times C_{n-1}(M\setminus *) \lra C_n(M\setminus *)$$
which adds a new point all around the missing point $*$. The long exact sequence on homology is then
$$\cdots \lra H_{*-d+1}(C_{n-1}(M \setminus *);\bF) \overset{\Delta_n}\lra H_*(C_n(M \setminus *);\bF) \lra H_*(C_n(M);\bF) \lra \cdots.$$

Thus to understand the dimension of $H_*(C_n(M);\bF)$ we need to understand the rank and nullity of the maps $\Delta_n$, and more precisely understand how they vary with $n$. The operations $\Delta_n$, and the maps $p_n$ defining them, come from a more general construction. The disjoint union $\mathcal{R} := \coprod_{n \geq 0} C_n(S^{d-1} \times [0,1])$ admits the structure of an $H$-space by putting configurations in cylinders end to end. For each manifold $M$, the disjoint union $\mathcal{M}(M \setminus *) := \coprod_{n \geq 0} C_n(M \setminus *)$ is an $H$-module over $\mathcal{R}$, by identifying a collar neighbourhood of the puncture in $M \setminus *$ with $S^{d-1} \times [0,1]$.

\begin{comment}

\begin{rem}
In fact much more is true: $\mathcal{R}$ admits the structure of an $A_\infty$-space and $\mathcal{M}(M\setminus *)$ admits the structure of an $A_\infty$-module over $\mathcal{R}$. It is then true that
$$\coprod_{n \geq 0} C_n(M) \simeq \mathcal{M}(M\setminus *) \otimes^{\bL}_{\mathcal{R}} \mathcal{M}(S^d\setminus *),$$
i.e.\ that the space of configurations on the closed manifold $M$ may be obtained as the two-sided bar construction of these two $\mathcal{R}$-modules. One may prove this rather easily, but the deeper reason is that the configuration space of a closed manifold is the so-called \emph{topological chiral homology} \cite[\S 4.1]{LurieTQFT} of $M$ with respect to the $E_d$-algebra $\coprod_{n \geq 0} C_n(\bR^d)$. The above equivalence is then an instance of the field-theoretic behaviour of topological chiral homology.
\end{rem}

\end{comment}

The maps $\Delta_n$ all come from the $H_*(\mathcal{R};\bF)$-module structure on $H_*(\mathcal{M}(M \setminus *);\bF)$, using the homology class of 
$$\Delta : S^{d-1} \lra C_1(S^{d-1} \times [0,1]) \lra \mathcal{R}.$$
The stabilisation map is given by taking the product with the homology class of
$$[1] : * \lra C_1(S^{d-1} \times [0,1]) \lra \mathcal{R}.$$
In order to understand the kernel and cokernel of the maps $\Delta_n$, we must understand how the maps $\Delta$ and $[1]$ interact.

\subsection{The commutator of $\Delta$ and $[1]$}

The operations $\Delta$ and $[1]$ both come from $H_*(\mathcal{R};\bF)$, so we may consider the square
\begin{diagram}
S^{d - 1} \times C_{0}(S^{d-1} \times \bR) &\rTo^{\Delta * -} & C_{1}(S^{d-1} \times \bR)\\
\dTo^{[1] * -} & & \dTo^{[1] * -}\\
S^{d - 1} \times C_{1}(S^{d-1} \times \bR) &\rTo^{\Delta * -} & C_{2}(S^{d-1} \times \bR)
\end{diagram}where the vertical maps add a new point at the left of the cylinder.

\begin{lem}\label{lem:Commutator}
On integral homology this diagram does not commute, and instead satisfies
$$\Delta*[1] = [1]*\Delta + \tau$$
where $\tau \in H_{d-1}(C_2(\bR^d);\bZ)$ is the Hurewicz image of the map $S^{d-1} \to C_2(\bR^d)$ given by $v \mapsto \{0, v\}$. We consider $\tau$ as a homology class on $C_{2}(S^{d-1} \times \bR)$ via the map
$$C_2(\bR^d)\lra C_{2}(S^{d-1} \times \bR).$$
\end{lem}
\begin{proof}
The two ways around the diagram give the $(d-1)$-cycles shown in the first part of Figure \ref{fig:Tubes}.
\begin{figure}[h]
\centering
\includegraphics[bb = 0 0 374 45]{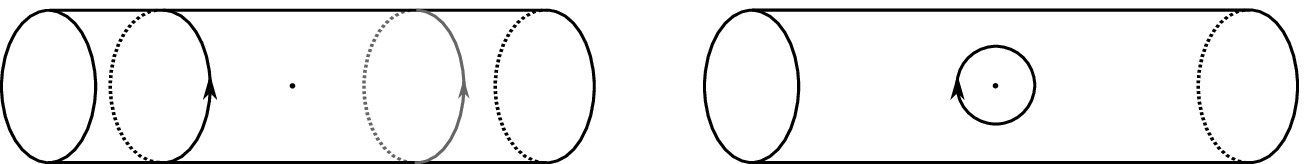}
\caption{}
\label{fig:Tubes}
\end{figure}
Their difference has an evident homology to the $(d-1)$-cycle shown in the second part of the figure, which is the homology class $\tau$.
\end{proof}

\subsection{The case $d$ odd or $\bF = \bF_2$}

Note that $C_2(\bR^d) \simeq \bR\bP^{d-1}$, and $\tau \in H_{d-1}(C_2(\bR^d);\bZ)$ is either zero or divisible by 2 as $d$ is odd or even. In particular it is trivial over $\bF_2$ or when $d$ is odd, and so the elements $[1]$, $\Delta \in H_*(\mathcal{R};\bF)$ commute in these cases. But then $[1]*-$ induces maps
$$\Ker(\Delta_n) \lra \Ker(\Delta_{n+1}) \quad\quad \Coker(\Delta_n) \lra \Coker(\Delta_{n+1})$$
which are isomorphisms in the stable range. This establishes Theorem \ref{thm:StableHomologyClosed} in this case.

\subsection{The case $d$ even over $\bQ$}

In this case the element $\tau$ is non-trivial, and so the argument of the last section does not work. Instead we will consider the square
\begin{diagram}
H_{*-d+1}(C_{n-1}(M\setminus *);\bQ) & \rTo^{\Delta * -} & H_*(C_n(M\setminus *);\bQ)\\
\dTo^{t_{n-1}} & & \dTo^{t_n} \\
H_{*-d+1}(C_{n-2}(M\setminus *);\bQ) & \rTo^{\Delta * -} & H_*(C_{n-1}(M\setminus *);\bQ)
\end{diagram}
involving the transfer maps instead of the stabilisation maps.

\begin{lem}
This digram commutes.
\end{lem}
\begin{proof}
The equation $t_n(\Delta * m) = \Delta * t_{n-1}(m) + t_1(\Delta) * m$ is seen to hold by considering the construction of the transfer: either the configuration point which is to be forgotten comes from $m$, or else it comes from $\Delta$. Commutativity of the diagram follows as $t_1(\Delta) \in H_{d-1}(C_0(M \setminus *);\bQ)=0$.
\end{proof}

By Proposition \ref{prop:Trf} the maps $t_i$ are rational isomorphisms in degrees $* \leq i-1$, and so the commutative diagram above gives that the kernel and cokernel of $\Delta_n$ have dimensions independent of $n$ in the stable range. The long exact sequence on homology for the cofibration sequence (\ref{eq:CofSeqClosed}) then implies that
$$\dim_\bQ H_*(C_n(M);\bQ) = \dim_\bQ H_*(C_{n-1}(M);\bQ)$$
for $* \leq n-1$. This establishes Theorem \ref{thm:StableHomologyClosed} in this case, and hence the first part of Theorem \ref{thm:MainClosed}.

\subsection{The transfer map}

Rather than simply knowing that the two rational vector spaces $H_*(C_n(M);\bQ)$ and $H_*(C_{n-1}(M);\bQ)$ are abstractly isomorphic, it is often useful to have a particular map giving the isomorphism. The following proposition establishes the last part of Theorem \ref{thm:MainClosed}.

\begin{prop}
The map $t_n : H_*(C_n(M);\bQ) \to H_*(C_{n-1}(M);\bQ)$ is an isomorphism for $* \leq n-1$.
\end{prop}
\begin{proof}
Unfortunately the cofibre sequence (\ref{eq:CofSeqClosed}) is of no use, as the connecting map
$$C_n(M) \lra S^d \wedge C_{n-1}(M \setminus *)$$
does not commute with transfers. Instead, we consider a new semi-simplicial resolution $D_n(M)^\bullet \to C_n(M)$. The space of $i$-simplices is 
$$D_n(M)^i := \{(C, p_0,...,p_i) \in C_n(M) \times M^{i+1} \,\, | \,\, p_j \notin C, \, p_j \neq p_k \}.$$
There are maps $d_j : D_n(M)^i \to D_n(M)^{i-1}$ given by forgetting the $j$-th point, and $D_n(M)^i \to C_n(M)$ given by forgetting all points.

The map $\| D_n(M)^\bullet \| \to C_n(M)$ is a fibre bundle, with fibre over a configuration $C$ given by $\| F(M \setminus C)^\bullet \|$ where the space $F(M \setminus C)^i$ is the space of $(i+1)$ ordered points in $M \setminus C$. As $C$ is non-empty, there is a $\bar{M} \subset M \setminus C$ which is homotopy equivalent and which misses some point $x$. Adding a new point at $x$ gives maps
$$F(\bar{M})^i \lra F(M \setminus C)^{i+1}$$
which determine a semi-simplicial nullhomotopy of the inclusion $\| F(\bar{M})^\bullet \| \to \| F(M \setminus C)^\bullet \|$. On the other hand, this inclusion is an equivalence, and hence $\| F(M \setminus C)^\bullet \|$ is contractible and $\| D_n(M)^\bullet \| \to C_n(M)$ is an equivalence.

We can take the same sort of resolution for $C_{n,1}(M)$, and we then get semi-simplicial maps
$$D_n(M)^\bullet \lla D_{n,1}(M)^\bullet \lra D_{n-1}(M)^\bullet.$$

We wish to define a semi-simplicial transfer map reversing the leftwards map, for which it will be convenient to have a particular geometric model for transfer maps of finite coverings. Let $SP^n(X)$ denote the $n$-fold symmetric product of a space $X$. If $\pi : E \to B$ is an $n$-fold cover, we obtain a map
$$\trf_\pi : B \to SP^n(E)$$
sending a point $b$ to $\pi^{-1}(b) \subset E$, a subset of cardinality $n$. The transfer map on homology is obtained by composing with the summation map
$$H_*(SP^n(E)) \lra H_*(E)$$
coming from the action $SP^\infty(SP^n(X)) \to SP^\infty(E)$ and the Dold--Thom theorem.

If we write $\tilde{C}_i(M)$ for the ordered configuration space of $i$ points in $M$, there are fibrations
$$C_n(M \setminus \text{$i+1$ points}) \lra D_n(M)^i \lra \tilde{C}_{i+1}(M).$$
Let $SP^n_{fib}(D_n(M)^i)$ denote the $n$-fold \emph{fibrewise} symmetric product with respect to this fibration. The map $D_{n,1}(M)^i \to D_n(M)^i$ is a fibrewise $n$-fold covering space over $\tilde{C}_{i+1}(M)$, so there is a transfer map
$$\trf : D_n(M)^i \lra SP^n_{fib}(D_{n,1}(M)^i).$$
The square
\begin{diagram}
D_n(M)^i & \lTo & D_{n,1}(M)^i\\
\dTo^{d_j} & & \dTo^{d_j}\\
D_{n}(M)^{i-1} & \lTo & D_{n,1}(M)^{i-1}
\end{diagram}
is a pullback, and hence the geometric transfer gives a commuting square
\begin{diagram}
D_n(M)^i & \rTo^\trf & SP^n_{fib}(D_{n,1}(M)^i)\\
\dTo^{d_j} & & \dTo^{d_j}\\
D_{n}(M)^{i-1} & \rTo^\trf & SP^n_{fib}(D_{n,1}(M)^{i-1})
\end{diagram}
for each $i$ and $j$, as $d_j$ gives a fibrewise map over the map $\tilde{C}_{i+1}(M) \to \tilde{C}_{i}(M)$ that forgets the $j$-th point. This determines maps of resolutions
\begin{diagram}
D_n(M)^\bullet & \rTo^\trf & SP^n_{fib}(D_{n,1}(M)^\bullet) & \rTo & SP^n_{fib}(D_{n-1}(M)^\bullet)\\
\dTo & & \dTo & & \dTo\\
D_{n}(M) & \rTo^\trf & SP^n(D_{n,1}(M)) & \rTo & SP^n(D_{n-1}(M))
\end{diagram}
and so a map of spectral sequences converging to the transfer map $t_n:H_*(C_n(M);\bQ) \to H_*(C_{n-1}(M);\bQ)$, which at $E_2^{s,t}$ is 
\begin{equation}\label{eq:SSTrfMap}
H_t(D_n(M)^s;\bQ) \lra H_t(D_{n-1}(M)^s;\bQ)
\end{equation}
induced by $D_n(M)^i \to SP^n_{fib}(D_{n,1}(M)^i) \to SP^n_{fib}(D_{n-1}(M)^i)$. The Serre spectral sequence converging to this map takes the form
$$H_p(\tilde{C}_{s+1}(M); \mathcal{H}_q(C_n(M \setminus \text{$s+1$ points}))) \lra H_p(\tilde{C}_{s+1}(M); \mathcal{H}_q(C_{n-1}(M \setminus \text{$s+1$ points})))$$
on $E_2^{p, q}$, where we have omitted the rational coefficients, and is induced by the transfer map
$$t_n: H_q(C_n(M \setminus \text{$s+1$ points});\bQ) \lra H_q(C_{n-1}(M \setminus \text{$s+1$ points});\bQ)$$
on coefficients. As the manifold $\{M \setminus \text{$s+1$ points}\}$ is open, this map is an isomorphism in degrees $q \leq n-1$ by Proposition \ref{prop:Trf}. Hence the map of spectral sequences (\ref{eq:SSTrfMap}) is an isomorphism for $t \leq n-1$, and so the map
$$t_n:H_*(C_n(M);\bQ) \to H_*(C_{n-1}(M);\bQ)$$
to which the spectral sequence converges is also an isomorphism for $* \leq n-1$, as required.
\end{proof}

\bibliographystyle{plain}
\bibliography{../MainBib}

\end{document}